\documentclass[10pt]{siamltex}
	\usepackage{epsfig}
	\usepackage{psfrag}
	\usepackage{a4wide}
	\usepackage{amsmath}
	\usepackage{tablefootnote}
	\usepackage{bm}
	\usepackage{amssymb}
	\usepackage{multirow}
	\usepackage{mdframed}
	\usepackage[dvipsnames]{xcolor}
	\usepackage{natbib}
	\usepackage{listings}

	\usepackage{subfigure}
	\usepackage{comment}
	\usepackage[bottom=2.85cm, top = 2.85cm, left = 2.5cm, right = 2.5cm]{geometry}
	
	\usepackage{algorithm}
	\usepackage{algorithmicx}
	\usepackage{algpseudocode}
\newcommand{\red}{}	
\newcommand{\p}{\texttt{$p$}}

	\usepackage{color}

\newcommand{\blue}{}

\usepackage{hyperref}
	\hypersetup{
     colorlinks = true,
     allcolors = blue,
     anchorcolor = green,
     citecolor = SeaGreen,
     filecolor = green,
     urlcolor = RubineRed
     }

\newcommand{\emi}{{\tt Emilia-923\,}}
\newcommand{\CC}{{\tt Cube}${5317k}\ $}

\newcommand{\fx}{\mathbf x}
\newcommand{\fw}{\mathbf w}
\newcommand{\fu}{\mathbf u}
\newcommand{\fz}{\mathbf z}
\newcommand{\fb}{\mathbf b}

\newcommand{\fv}{\mathbf v}

\newcommand{\fr}{\mathbf r}

	\title{Parallel Newton-Chebyshev Polynomial Preconditioners \\ for the COnjugate Gradient method}
	\author{L. Bergamaschi\footnotemark[2]
	 \and A. Mart\'{\i}nez\footnotemark[4]}

\usepackage{fancyhdr}
\begin{document}
\lhead{L. Bergamaschi and A. Mart\`{\i}nez}
\rhead{Polynomial preconditioners for the PCG method}

	\thispagestyle{plain}
	\date{\today}
	\maketitle

	\renewcommand{\thefootnote}{\fnsymbol{footnote}}

	\footnotetext[1] {Corresponding author.}
	\footnotetext[2] {Department of Civil Environmental and Architectural Engineering,
			  University of Padova, Italy, \href{mailto:luca.bergamaschi@unipd.it}{luca.bergamaschi@unipd.it}.}
	\footnotetext[4] {Department of Mathematics and Earth Sciences,
			  University of Trieste, Italy, \href{mailto:amartinez@units.it}{amartinez@units.it}.}

	\renewcommand{\thefootnote}{\fnsymbol{footnote}}

	\overfullrule=0pt

	\begin{abstract}
	In this note we exploit polynomial preconditioners for the Conjugate Gradient method 
	to solve large symmetric positive definite linear systems in a parallel environment.
	We put in connection a specialized Newton method to solve the matrix equation $X^{-1} = A$
		and the Chebyshev polynomials for preconditioning. We propose
		a simple modification of one parameter which avoids clustering of extremal
		eigenvalues in order to speed-up convergence. We provide
		results on very large matrices (up to 8 billion unknowns in a parallel
		environment) showing the efficiency of the proposed class of preconditioners.
	\end{abstract}

	\section{Introduction}

	Discretization of PDEs modeling different processes and constrained/unconstrain\-ed optimization problems
	often require the repeated solution of 
	large and sparse linear systems $A \fx = \fb$, {\blue in which $A$ is symmetric positive definite}.
	The size of these system can be of order $10^6 \div 10^9$ and this calls for the use of
	iterative methods, equipped with ad-hoc preconditioners as accelerators running on a  parallel computing environment.
	In most cases the huge size of the matrices involved prevents their complete storage. In these 
	instances only the application of the matrix to a vector is available as a routine (\textit{matrix
	-free regime}).
	Differently from direct factorization methods, 
	iterative methods do not need the explicit knowledge of the coefficient matrix. The issue is the
	construction of a preconditioner which also work in a matrix-free regime.
	The most common (full-purpose) preconditioner such as the incomplete LU factorization or most of the approximate 
	inverse preconditioners rely on the knowledge of the coefficients of the matrix. An exception
	is represented by the AINV preconditioner (\cite{MR1787297}), whose construction is however
	inherently sequential. In all cases 
	 factorization based methods are not easily parallelizable, the bottleneck being
	the solution of triangular systems needed when they are applied to a vector.

	Polynomial preconditioners, i.e. preconditioners that can be expressed as $P_k(A)$, are very
	attractive for the following main reasons:
	\begin{enumerate}
		\item Their construction is only theoretical, namely only the coefficients
			of the polynomial are to be computed with negligible computational cost.
		\item The application of $P_k(A)$ require a number, $k$, of matrix-vector products
			so that they can be implemented in a matrix-free regime.
		\item The eigenvectors of the preconditioned matrix are the same as those
			of $A$.
	\end{enumerate}

The use of polynomial preconditioner for accelerating Krylov subspace methods is not new.
We quote for instance the initial works in \cite{MR694525,doi:10.1137/0906059} to accelerate the Conjugate Gradient method 
and \cite{VANGIJZEN199591} where polynomial preconditioners are used
		to accelerate the GMRES \cite{SaadSchultz86} method.

However, these ideas have been recently resumed, mainly in the context
of nonsymmetric linear systems, e.g. in  \cite{loe2019new,loe2019polynomial}
	or in the  acceleration of the Arnoldi method for eigenproblems \cite{embree2018polynomial}.
	An interesting contribution to this subject is the work in \cite{kaporin} where Chebyshev-based polynomial
	preconditioners are applied in conjunction with sparse approximate inverses.

		The aim of this paper is twofold.
		We first give a theoretical evidence that a polynomial preconditioner for the CG method
 can be developed by starting from the well-known Newton's method to solve
 the matrix equation $P^{-1} - A = 0$. We will show that with a simple modification this method
 reveals equivalent, in exact arithmetics, to the Chebyshev polynomial preconditioner.
 The second objective of this paper is to show that polynomial preconditioners of very high degree
 can be useful to cut down the number of scalar products and improve 
 consistently the parallel scalability of the PCG method. Minimizing scalar products within
 Krylov subspace solvers is currently a matter of research (see e.g. the recent work in \cite{exascale}).

 The rest of the paper is organized as follows: In Section \ref{Newton} we  develop a recursion
 for preconditioners based on the Newton formula. In Section \ref{ChebSec} we review the theory 
 regarding Chebyshev polynomial preconditioners and show the
 equivalence between the Newton recurrence and a non standard recurrence for Chebyshev polynomials.
 A strategy to avoid clustering of the eigenvalues near the end of the spectrum which greatly enhances the performance of the proposed preconditioners is described
 in Section \ref{scale}.
 In Section \ref{numris} we report numerical results on both sequential and parallel computing environments obtained in the solution
 of very large linear systems (up to $8 \times 10^9$ unknowns for the
 largest problem) which we use as tests for our preconditioned CG.
 In Section \ref{conc} we draw some conclusions and propose topics for future research on the subject.

\newcommand{\argmin}{\text{arg}\!\min}
\section{Newton-based preconditioners}
\label{Newton}
The Newton preconditioner can be obtained as a trivial application of the Newton-Raphson method to the scalar equation
\[ x^{-1} - a = 0, \quad a \ne 0,\]
which reads 
\[ x_{j+1} = 2 x_j - a x_j^2, \quad j =0,\ldots, \qquad x_0 \ \text{fixed}.\]
The matrix counterpart of this method applied to $P^{-1} - A = 0$ can be cast as
\begin{equation}
	\label{newtonP}
 P_{j+1} = 2 P_j - P_j A P_j, \quad j = 0, \ldots, \qquad P_0 \text{ fixed}, 
\end{equation}
which is a well-known iterative method for matrix inversion (also known as Hotelling's method
\cite{hotelling1943}).

If $P_0$ is a given preconditioner for $A$ satisfying $P_0 A = A P_0$, then $\{P_j\}$ can be seen as a  sequence of preconditioners
converging to $A^{-1}$ if $\|I - P_0 A\|  = r < 1$. In fact, denoted by $E_j = I - P_j A$ we have
that {\red $\|E_j\|\le r^{2^{j}}$} as it can be easily proved by induction:
\[ \|E_{j+1}\| = \| I - 2 P_j A + (P_j A)^2\|  =\| E_{j}^2\| \le \|E_{j}\|^2  {\blue \le} {\red (r^{2^{j}})^2 =  r^{2^{j+1}}} \]
which implies $\displaystyle \lim_{j \to \infty} \|E_j\| = 0$.

Sequence $\{P_{j} \}$ can not be explicitly formed since it would produce
increasingly dense matrices. Actually, inside the PCG method only the product of $P_j$ times
a vector is needed and hence recursively we
have  \[\fw = P_{j+1} \fr \Longleftrightarrow \left \{\begin{array}{l}
\fu = P_j \fr \\ \fv = A \fu \\ \fw = 2\fu -P_j \fv\end{array} \right . \]
This method, as it is, is never used to form a preconditioner as it requires doubling
the computational work per iteration, while the condition number is reduced by a factor less than
4.  In fact, the condition
$\|I - P_0 A\| < 1$, with $P_0 A$ symmetric, is equivalent to the condition $0 < \lambda(P_0 A) < 2$.
Hence, assuming $1 \in \sigma(P_0 A$) 
the eigenvalues of $P_{1} A = 2 P_0 A - (P_0 A)^2 $ map a generic eigenvalue $\mu$
of $P_0 A$ in $2\mu - \mu^2$ with 
\[ \begin{cases} 
	\mu_{\min} &\mapsto 2\mu_{\min} - \mu_{\min}^2 \le 2 \mu_{\min} \\
	\mu_{\max} &\mapsto 2\mu_{\max} - \mu_{\max}^2 \le 1 \\
	 1 & \mapsto 1 \end{cases} \]
with $\kappa(P_1 A) \ge \dfrac{1} {2 \mu_{\min}} > \dfrac{\kappa(P_0 A)}{4}$.
In the next step, however, as the eigenvalues
of $P_1 A$ now lie in the interval $[\mu_1, 1]$, 
they are approximately mapped into $[2\mu_1, 1]$ with the condition number only halved.
Due to the asymptotic Conjugate Gradient convergence bounds,  
a halving of the condition number would imply a 1.4
reduction in the iteration number, the cost of a single iteration being doubled.

The efficiency of such a Newton method can however be increased
due to the following result:
\begin{theorem}
	\label{newtTh}
	Let $\alpha_j, \beta_j$ be the smallest and the largest eigenvalues of $P_j A$.

	If $0 < \alpha_j <  1  < \beta_j \le 2 - \alpha_j$ then $[\alpha_{j+1}, \beta_{j+1}]
	\subset [2 \alpha_j - \alpha_j^2, 1]$.
\end{theorem}
\begin{proof}
	Every eigenvalue of $P_{j+1} A $, $\lambda_i^{(j+1)}$ satisfies $\lambda_i^{(j+1)} = 
	f(\lambda_i^{(j)}) $ where the function  $f(t) = 2t - t^2$
	maps the interval $[\alpha_j, 2 - \alpha_j]$ into
	$[f(\alpha_j), 1]$.
\end{proof}

\noindent
If $\beta_j = 2-\alpha_j$ then the reduction in the condition number from $P_j A$ to $P_{j+1}A$ is near 4
provided that $\alpha_j$ is small:
\[ \frac{\kappa(P_j A)}{\kappa(P_{j+1} A)} = \frac{2-\alpha_j}{\alpha_j} (2\alpha_j - \alpha_j^2)
= (2 - \alpha_j)^2 \approx 4.\]
Under these hypotheses each Newton step provides an average halving of the CG iterations (and hence
of the number of scalar products) as
opposed to twice the application of both the coefficient matrix and the initial preconditioner. 
This idea  can be efficiently employed when $P_0 = I$ to cheaply obtain a {\bf polynomial preconditioner}.
{\blue This also includes diagonal preconditioning since the original linear system, $\hat A \hat \fx = \hat \fb $ can be
symmetrically scaled by the diagonal of $\hat A$, $D = \text{diag}(\hat A)$ obtaining  the system
$A\fx = \fb$ where $A = D^{-1/2} A D^{-1/2}$.}

At the first Newton stage the preconditioner must be scaled by $\zeta_0 = \dfrac{2}{\alpha_0+\beta_0}$ in order to 
satisfy the hypotheses of Theorem \ref{newtTh}. Hence
the eigenvalues of 
$P_1 A = \left(2 \zeta_0 I - \zeta_0^2 A\right) A$ will lie in $[\alpha_1, \beta_1] $
where $\beta_1 = 1$ and $\alpha_1 = (2-\alpha_0 \zeta_0) \alpha_0 \zeta_0$ and the next scaling factor will
be $\zeta_1 = \dfrac{2}{1 + \alpha_1}$.
Analogously, at a generic step $j >1$, $\alpha_j = (2-\alpha_{j-1} \zeta_{j-1}) \alpha_{j-1} \zeta_{j-1}$ and
$\zeta_j = \dfrac{2}{\alpha_j + 1}$. Finally, exploiting the relation $\alpha_{j-1} \zeta_{j-1}
= 2 - \zeta_{j-1}$ we can write
\begin{equation}
	\label{zetaNewt}
 \zeta_j = \dfrac{2}{1+ \zeta_{j-1} (2-\zeta_{j-1})} = 
	    \dfrac{2}{1+ 2\zeta_{j-1}-\zeta_{j-1}^2} .
\end{equation}
{\red 
Then the recurrence for the preconditioners is obtained from (\ref{newtonP}) by scaling $P_j$ with
$\zeta_j$ as
\begin{eqnarray}
	\label{newrec}
	P_{j+1} &=& 2 \zeta_j P_j - \zeta_j^2 P_j A P_j, \quad j = 0, \ldots, \qquad P_0 = I
\end{eqnarray}
}
\noindent
This suggests an analogous recurrence for the polynomials of degree $k=2^j-1, j = 0, \ldots$ as
\begin{eqnarray*}
	p_0(x) &=& 1 \nonumber \\
	p_{2^{j+1}-1}(x) &=&  {\red 2} \zeta_j p_{2^j-1}(x) - \zeta_j^2 x\, p_{2^j-1}^2(x), \quad j = 0, \ldots, 
\end{eqnarray*}
Finally, setting $r_{2^j-1}(x) = \zeta_j p_{2^j-1}(x)$ we can write a slightly more efficient recursion, as
\begin{eqnarray}
	\label{recpolr}
	r_0(x) &=& \zeta_0 \nonumber \\
	r_{2^{j+1}-1}(x) &=&  \zeta_{j+1} \left({\red 2} r_{2^j-1}(x) - x\, r_{{2^j-1}}^2(x)\right), \quad j = 0, \ldots. 
\end{eqnarray}
\begin{algorithm}[h!]
	\begin{algorithmic}[1]
	\caption{Newton-based polynomial preconditioner}
		\label{NewtAlg}
		\State Approximate the extremal eigenvalues of $A$: $\alpha_0, \beta_0$.
		\State Set the number of Newton steps: \texttt{nlev}  
		\smallskip
		\State Set $\zeta_0 = \dfrac{2}{\alpha_0 + \beta_0}, \quad  
		\zeta_1 = \dfrac{2}{1+ 2\alpha_0\zeta_{0}-(\alpha_0\zeta_{0})^2},  \quad
		\zeta_i = \dfrac{2}{1+ 2\zeta_{i-1}-\zeta_{i-1}^2}, \quad i = 2, \texttt{nlev}.$  
		\smallskip
    \State Solve $A\fx = \fb$ by CG accelerated with the polynomial preconditioner $P_{\texttt{nlev}}$.

		\State Recursive application  of $P_{\rm nlev}$ to a vector   $\fu$ at each PCG iteration
			\begin{eqnarray}
				\label{recursive}
				P_0 \fu & = & \zeta_0 \fu \nonumber \\
				P_{j+1} \fu &=&  \zeta_{j+1} \left(2 P_j \fu - P_j A P_j\fu\right) , \qquad
				j = \texttt{nlev} -1, \ldots, 0
			\end{eqnarray} 
	\end{algorithmic}
\end{algorithm}

Our polynomial preconditioner is then defined as $P_j = r_{2^j-1}(A)$. Its application
to a vector, in view of (\ref{recpolr}) is described in Algorithm \ref{NewtAlg}.

We also provide in Figure \ref{matlabalg} the very simple Matlab function for the application of the preconditioner
within the PCG procedure.
\smallskip

\begin{figure}[h!]
	\begin{mdframed}
\begin{lstlisting}
function p_res = applyrec(zeta,nlev,A,res)
if nlev > 0
   u =applyrec(zeta,nlev-1,A,res);
   v = A*u;
   w = applyrec(zeta,nlev-1,A,v);
   p_res = zeta(nlev)*(2 u - w);
else
   p_res= zeta(1)*res;
end
\end{lstlisting}
	\end{mdframed}
	\caption{Matlab recursive function for the application of the Newton-based polynomial preconditioner}
	\label{matlabalg}
\end{figure}

\smallskip
\noindent

\section{Chebyshev preconditioners}
In this Section we recall the main steps to arrive at the iterative definition of the polynomial
preconditioner based on the Chebyshev polynomials of the first kind. More details
can be found in \cite{Saad03}.
\label{ChebSec}
The optimal polynomial preconditioner $q_k(x)$ for the CG method should minimize the condition number of 
$P_k A$ for a given degree $k$. This problem can be formulated as
\[ \text{Find} \  p_k \in \Pi_k \  \text{such that} \   p_k = \argmin_{\substack{\hspace{-4mm} p_k \in \Pi_k}} \max _{\lambda \in \sigma(A)} |1 - p_k(\lambda) \lambda|, \]
where $\Pi_k$ is the set of polynomials of degree $k$ at most.
Since this problem can not be solved without knowing all the eigenvalues of $A$, it is
replaced by the following problem
\begin{equation}
	\label{maxmin}
 \text{Find} \  p_k \in \Pi_k \  \text{such that} \  p_k =
	\argmin_{\substack{\hspace{-4mm} p_k \in \Pi_k}} 
	\max _{\lambda \in I} |1 - p_k(\lambda) \lambda| = 
	\argmin_{\substack{\hspace{-5mm} q_{k+1} \in \Pi_{k+1}\\ \hspace{-5mm} q_{k+1}(0) = 1}} 
	\max _{\lambda \in I} |q_{k+1}(\lambda) |
\end{equation}
where $q_{k+1}(x) = 1-xp_k(x)$ and  $I = [\alpha, \beta] \supset [\lambda_1, \lambda_n]$, whose solution requires an approximate
knowledge of the extremal eigenvalues of $A$.
The polynomial  that solves (\ref{maxmin}) is the shifted and scaled Chebyshev polynomial 
of degree $k+1$ \cite{cheney}
\begin{equation}
\label{cheb}
 q_{k+1}(x) = \frac{T_{k+1}\left(\frac{\alpha + \beta - 2x}{\beta-\alpha}\right)}{T_{k+1}\left(\frac{\alpha+\beta}{\beta-\alpha}\right)}. \end{equation}
The wanted optimal polynomial for preconditioning 
is therefore $p_k(x) = x^{-1} \left(1-q_{k+1}(x)\right)$. 
Exploiting the well-known three-term recursion for the Chebyshev polynomials:
\begin{equation}
	\label{3term}
	T_{k+1}(x) = 2x T_k(x) - T_{k-1}(x), \qquad T_1(x) = x, \qquad T_0(x) = 1,
\end{equation}
we can develop a recurrence also for the polynomials $\{p_k(x)\}$.
We set \[\theta = \frac{\beta+\alpha}{2}, \quad 
	 \delta = \frac{\beta-\alpha}{2}, \quad \text{and} \quad \sigma = \frac{\theta}{\delta} \]
	 so that we can rewrite (\ref{cheb}) as
	\begin{equation} 
	\label{qT} 
	 q_{k+1}(x) = \frac{T_{k+1}\left(\sigma - \frac{x}{\delta}\right)}{T_{k+1}(\sigma) }
		= \frac{T_{k+1}\left(\sigma - \frac{x}{\delta}\right)}{\sigma_{k+1}}, \qquad
		\text{with} \ \sigma_{k+1} = T_{k+1}(\sigma)
	\end{equation} 
The $q_k$'s satisfy a recursion analogous to (\ref{3term})  as:
\begin{equation}
\label{q}
 q_{k+1}(x) = \frac{1}{\sigma_{k+1}} \left(2(\sigma - \frac{x}{\delta}) 
	\sigma_k q_k(x) - \sigma_{k-1} q_{k-1}(x) \right) , \quad q_1(x) = 1 -\frac{x}{\theta}, \quad q_0(x) = 1. 
\end{equation}
Noticing that the denominator of (\ref{qT})  satisfies the recursion, for $k \ge 1$,
\[\sigma_{k+1} = 2 \sigma \sigma_k - \sigma_{k-1}, \quad \sigma_1 = \sigma, \quad \sigma_0 = 1, \]
and defining $\rho_k = \dfrac{\sigma_k}{\sigma_{k+1}}$ we rewrite  (\ref{q}) as
\begin{equation}
\label{q1}
	q_{k+1}(x) = \rho_{k} \left(2\left(\sigma - \frac{x}{\delta}\right) q_k(x) - \rho_{k-1} q_{k-1}(x) \right) 
\end{equation}
with \begin{equation}
	\label{rhocheb}
	\rho_k = \dfrac{1}{2 \sigma - \rho_{k-1}}, \ k \ge 1 \quad \text{and} \quad \rho_0 = \dfrac{1}{\sigma}.
\end{equation}
To obtain an explicit expression for our preconditioner it remains to develop a recursion
for the sequence of polynomials $\{p_k(x)\}$. To this aim we write $q_k(x)$ in terms of $p_k(x)$ as
$q_{k+1}(x) = 1 - x p_k(x)$ and substitute this expression into (\ref{q1}) obtaining $p_{-1}(x) = 0$,
$p_0(x) = \dfrac{1}{\theta}$ and, for $k \ge 1$,
\[
	1 - x p_{k}(x)  = 
	\rho_{k} \left(2\left(\sigma - \frac{x}{\delta}\right) (1-x p_{k-1}(x)) - \rho_{k-1} (1-x p_{k-2}(x)) \right).  \]
From which  we obtain the recursion (see e.g. \cite{MR2169217}) 
\begin{eqnarray*}
	p_{-1}(x) & = &  0 \nonumber \\
	p_0(x) & = &  \frac{1}{\theta} \nonumber \\
	p_k(x) & = & \rho_k\left(2 \sigma \left(1 - \frac{x}{\theta}\right) p_{k-1}(x)
	- \rho_{k-1} p_{k-2}(x) + \frac{2}{\delta}\right), \qquad k \ge 1.
\end{eqnarray*}
The application of the Chebyshev preconditioner of degree $m$,  $P_m = p_m(A)$ within the PCG solver is 
described in Algorithm \ref{ChebAlg}.

\begin{algorithm}
	\begin{algorithmic} [1]
		\caption{Computation  of the preconditioned residual $\hat \fr = P_m \fr$ with Chebyshev preconditioner.}
		\label{ChebAlg}
		\State Compute $\rho_k, k = 1,\ldots, m_{\max}$ using (\ref{rhocheb})
		\State $\fx_{old}  = \fr/\theta$ \quad \hspace{12.7mm} (\textit{if {m} $=0$ exit with 
		$\hat \fr = \fx_{old}$})
		\smallskip

		\State $\fx = \dfrac{2\rho_1}{\delta} \left(2\fr - \dfrac{A \fr}{\theta}\right)$
		\hspace{2.7mm}
		(\textit{if {m} $=1$ exit with 
		$\hat \fr = \fx$})

		\For {$k =2: {m}$}
		\smallskip
\State		$\fz = \dfrac{2}{\delta} \left(\fr - A \fx\right)$
\smallskip

\State		$\hat \fr = \rho_{k+1}\left(2\sigma \fx- \rho_k\fx_{old}+\fz\right)$
		\State          $ \fx_{old} = \fx; \ \fx = \hat \fr$.
\EndFor


	\end{algorithmic} 
\end{algorithm}
\subsection{Other recursions}
The algorithm for the Chebyshev preconditioner can be greatly simplified by taking into account
the following relation involving Chebyshev polynomials:
\[ T_{2k}(x) = 2 T_k^2(x) - 1.\]
Proceeding as before we can define a recursion for the shifted and scaled polynomials as:
{\red \[ q_{2k}(x) = \frac{1}{\sigma_{2k}} \left(2 \sigma_k^2 q_k^2(x) - 1\right) \]}
where $\sigma_{2k} = 2\sigma_k^2 -1$, and finally a formula for the $p_k$'s as:
\begin{equation}
	p_{2k-1} (x) = \frac{2 \sigma_{k}^2}{\sigma_{2k}} \left(2 p_{k-1}(x) -  x p_{k-1}^2(x)\right),
	\quad k \ge 1, \qquad p_0(x) = \frac{1}{\theta} 
	\label{recur}
\end{equation}
which resembles formula  (\ref{recursive}). Actually the two formulae are mathematically equivalent
as proved in the following Theorem
\begin{theorem}
	Let  $\chi_{j} = \dfrac{2 \sigma_{k}^2}{\sigma_{2k}}, \ j = \log_2 k$, then the sequence  (\ref{recur})
satisfies the relation (\ref{recpolr}).
\end{theorem}
\begin{proof}
	We show that the polynomials $p_{2^j-1}(x)$ defined by the recurrence (\ref{recur})
	coincide with the polynomials $r_{2^j-1}(x)$ of (\ref{recpolr}).  
	As  $p_0 = r_0$, it is sufficient to prove that 
	\[\zeta_{j}  = \dfrac{2\sigma_k^2}{\sigma_{2k}} \equiv \chi_{j}, \qquad j \ge 1, \qquad 
	k = 2^j. \]
	First, observe that $\dfrac{1}{\sigma} = 1 - \dfrac{\alpha}{\theta} = 1 - \alpha_0 \zeta_0$, then  
	\[ \chi_1 = 
	 \frac{2\sigma^2}{2\sigma^2-1} = \frac {2}{2 - (\sigma^{-1})^2} 
	  = \frac {2}{1 + 2 \alpha_0 \zeta_0 - \alpha_0^2 \zeta_0^2} = \zeta_1.\] 
	Finally, for $j > 1$,
	\[ \chi_{j+1} =  \dfrac{2 \sigma_{k}^2}{\sigma_{2k}} 
	=  \dfrac{1+\sigma_{2k}}{\sigma_{2k}}   = \frac{1}{\sigma_{2k}} + 1 \quad
	\Longrightarrow \quad \sigma_{2k} = \frac{1}{\chi_{j+1}-1} \quad (\text{and hence } 
	\sigma_{k} = \frac{1}{\chi_{j}-1}) 
	.\]
	Then 
	\[  \chi_{j+1} =  \frac{1}{\sigma_{2k}} + 1 = 
	\frac{1}{2\sigma_k^2-1} +1 =
	\frac{2\sigma_k^2}{2\sigma_k^2-1} =
	\frac{2}{2-(\sigma_k^{-1})^2} =
	\frac{2}{2-(\chi_j-1)^2} =
	\frac{2}{1 + 2\chi_j - \chi_j^2},\]
	which is the (\ref{zetaNewt}).
\end{proof}

We have proved that the scaled Newton polynomials and the Chebyshev polynomials
are the same. One can use either the recursive version (Algorithm \ref{NewtAlg}) or the 
iterative version (Algorithm \ref{ChebAlg}) with no difference in exact arithmetics.
Due to this equivalence we will call our preconditioner: Newton-Chebyshev (NC in short) polynomial
preconditioner.
\section{The optimal parameters are not optimal}
\label{scale}
Supposing  that the extremal eigenvalues are exactly known, the best performance of the PCG
method is not necessarily  achieved when the condition number of the preconditioned matrix is minimized.
Actually the NC polynomial preconditioner, while reducing the spectral interval
and the condition number of $P(A) A$ provides a clustering of the extremal eigenvalues.
\begin{figure}[h!]
	\begin{center}
	\begin{minipage}{7.4cm}
	\includegraphics[width=7.4cm]{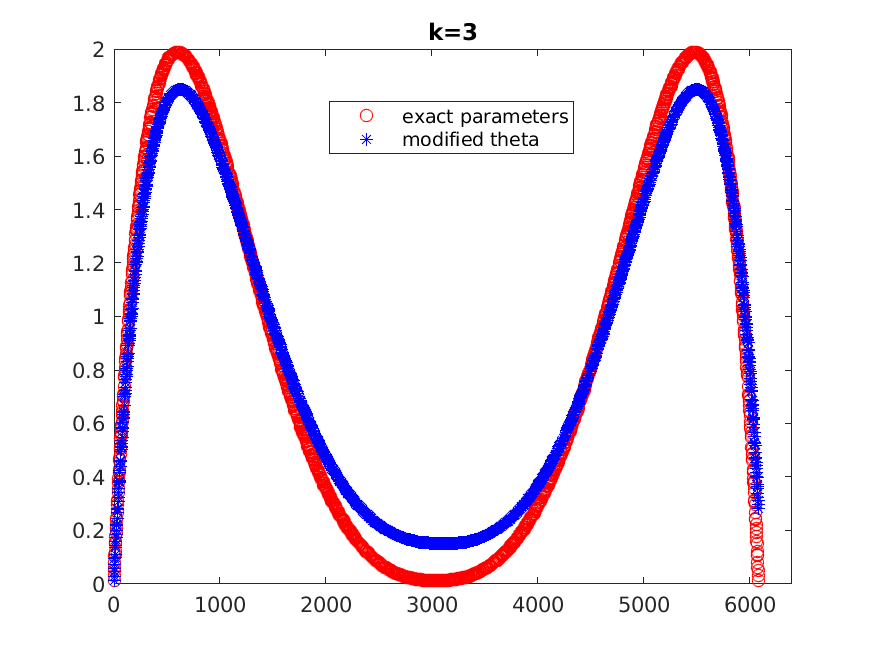}
	\end{minipage}
	\begin{minipage}{7.4cm}
	\hspace{-1mm}
	\includegraphics[width=7.4cm]{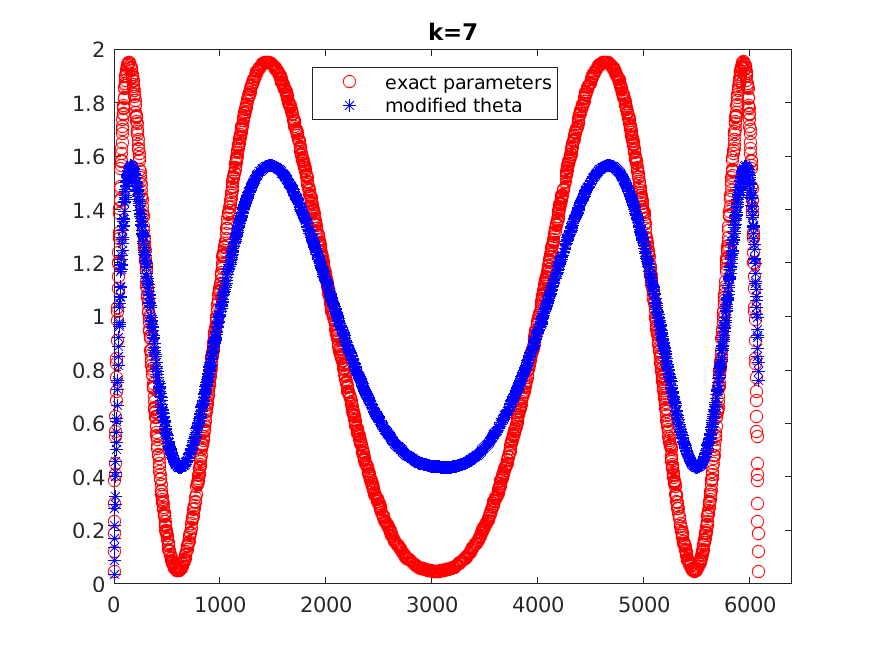}
	\end{minipage}

	\begin{minipage}{7.4cm}
	\includegraphics[width=7.4cm]{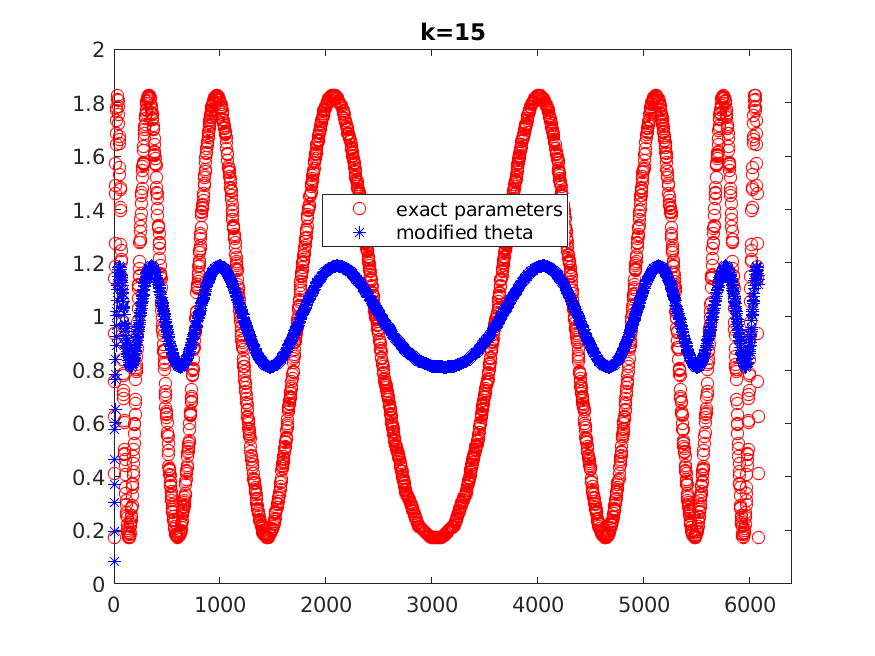}
	\end{minipage}
	\begin{minipage}{7.4cm}
	\hspace{-1mm}
	\includegraphics[width=7.4cm]{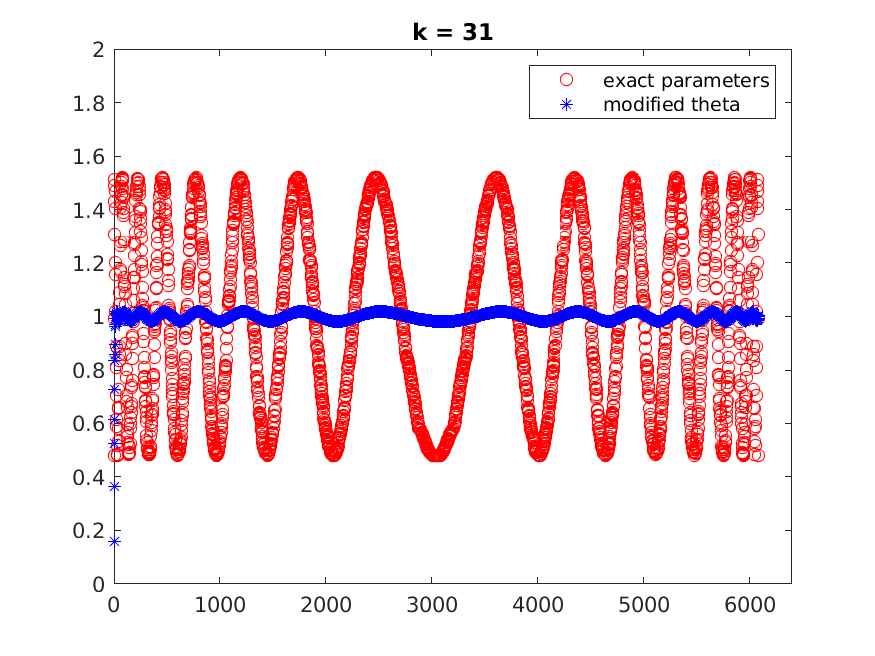}
	\end{minipage}
	\caption{Eigenvalue distribution of $P_k(A) A$ using the exact parameters (red circles) and
	modified {$\theta$}-value (blue stars) for different polynomial degrees} 
	\label{eigdis}
	\end{center}
\end{figure}

To clarify the situation we constructed the exact Chebyshev polynomials 
for the FD discretization of the Laplacian matrix in the unitary square
of size 6084 whose exact eigenvalues are known. In Figure \ref{eigdis} we provide the eigenvalue
distribution (red circles) of the preconditioned matrix $P_k(A) A, k = 3, 7, 15, 31.$
In the same picture we also provide the same plots, in which, however, the  initial value of 
$\theta$
has been slightly modified by
multiplying it by $1.01$
(the same result would have been obtained by reducing $\zeta_0 = \dfrac{1}{\theta}$ in the Newton-based
approach).  The eigenvalue distribution is represented with blue stars in this case.  
{\blue 
The meaning of the figure is as follows:
		a circle/star with coordinate $(s,y)$  represents an eigenvalue of the
		preconditioned matrix $p_k(A) A$, namely
		$ y = \lambda_s p_k(\lambda_s)$, where $\lambda_s$ is the $s$-th eigenvalues
		of $A$ in increasing order.}

Employing the Chebyshev preconditioner with exact parameters, the condition number is minimized but
a clear clustering of the smallest eigenvalues is produced ({\blue see the bottom part of the red plots 
in Figure \ref{eigdis}} and
 also Table \ref{comp80}, where 
the values of the indicator $l$, defined in (\ref{lmin}), are shown).
Slightly increasing the parameter $\theta$ yields an asymmetric spectrum of the preconditioned matrices
which avoids clustering especially of the smallest eigenvalues which are very well separated.
This behavior is known to speed-up the PCG convergence.

\begin{table}[h!]
	\caption{PCG iterations for solving the $78^2$ discretized Laplacian in the unit square
with polynomial preconditioner of degree $0, 1, 3, \ldots, 31$. The extremal eigenvalues, the number $l$ of eigenvalues close to the minimum 
and the condition number of the preconditioned matrices are also reported.}
	\label{comp80}
	\begin{center}
	\begin{tabular}{rrrcrr|rrcrr}
		&		 \multicolumn{5}{c}{Original NC algorithm}      
		& \multicolumn{5}{c}{NC with  $\theta$ scaled by $1.01$} \\
		$m$ & iter & $\mu_{\max}$   & $\mu_{\min}$ & $l$ & $ \kappa (P_m A)$ &
		iter & $\mu_{\max}$ &  $\mu_{\min}$ &$l$ &  $ \kappa (P_m A)$ \\
		\hline
		0    & 223   & 1.9992 &  7.9060e-04 &1&  2528.7&223&    1.9794 &  7.8278e-04 &1&   2528.7 \\
		1    & 111   & 1.9968 &  3.1562e-03 &2&   632.7&112&   1.9584  & 3.0647e-03  &1&   639.0 \\
		3    & 115   & 1.9875 &  1.2526e-02 &188&   158.7&61&  1.8493   &1.1318e-02   &1&  163.4 \\
		7    &  58   & 1.9514 &  4.8580e-02 &278&    40.2&31& 1.5640   &3.5202e-02   &1&   44.4 \\
		15    &  30   & 1.8268 &  1.7318e-01 &468&    10.5&17&1.1891   &8.2247e-02    &1&  14.5 \\
		31    &  15   & 1.5193 &  4.8067e-01 &874&     3.2&11&1.0182  & 1.6060e-01    &1&   6.3 \\
	\end{tabular}
	\end{center}
	\end{table}

Indeed in Table \ref{comp80} the reported results of the run for  the polynomial preconditioners of degree
$2^j-1,  j = 0, \ldots, 5$ confirm that the scaling the Newton-Chebyshev polynomial
preconditioner highly improves its  performance as compared to using the \textit{optimal} parameters.
In the same Table we report the number of eigenvalues of the preconditioned matrix which are close
to the minimum as 
\begin{equation}
	\label{lmin}
l = \# \left\{\lambda \, :  \, \frac{\lambda}{\lambda_{\min}} < 1.1\right\} .
	\end{equation}
With the scaled NC algorithm the smallest eigenvalue is isolated while
with optimal parameters the number $l$ increases with the degree of the polynomial.
\section{Numerical Results}	
\label{numris}
We now report the results of numerical experiments to solve very large and sparse matrices,
most of them arising
from real engineering applications. In detail,
\begin{itemize}
	\item \texttt{Opt\_Transp} arises from the Finite Element discretization of the transient optimal transport problem \cite{BFMPJcam17}.
	\item \texttt{Lap1600}: is the Laplacian on the unitary square with $1598^2$
interior grid points.
	\item {\tt Cube\_5317k}: arises from the equilibrium of a concrete cube discretized by a
regular unstructured tetrahedral grid.
\item {\tt Emilia\_923}:  arises from the regional geomechanical model of a deep hydrocarbon
	reservoir~\cite{FERRONATO20101918}.  It is obtained discretizing the  structural problem
 with tetrahedral Finite Elements.
Due to the complex geometry of the geological formation it was
not possible to obtain a computational grid characterized by
regularly shaped elements.
\end{itemize}
		The size and nonzero numbers of these problems are reported in Table \ref{size}.
		\begin{table}[h!]
			\caption{Size $n$, number of nonzeros \texttt{nnz} and spectral condition number $\kappa = \dfrac{\lambda_{n}}{\lambda_{1}}$  (computed
			after symmetric diagonal scaling) of the test matrices. Regarding matrix  \texttt{Opt\_Trans} the condition number has been computed
			as $\kappa = \dfrac{\lambda_{n}}{\lambda_{2}}$ since $\lambda_1 = 0$.}
\label{size}
\begin{center}
\begin{tabular}{|l|r|r|c|}
\hline
	name & $n$ &  \texttt{nnz} & $ \kappa(A)$  \\
\hline
	\texttt{Opt\_Trans} & 412417 & 2\,882817  & 1.04 $\times 10^6$\\
	\texttt{Lap1600} & 2\,553604 & 12\,761628 & 2.42 $\times 10^5$\\
	\emi & 923136 & 41\,005206  & 3.08 $\times 10^5$\\
	\CC &  5\,317443 & 222\,615369 & 3.30 $\times 10^6$\\
\hline
\end{tabular}
\end{center}
		\end{table}

    In the following results we will employ a polynomial of degree $m = 2^{\rm nlev}-1$,
    with various values of the parameter {\texttt nlev} which also counts the Newton iterations.
The scaling factor was set to $1.001$ for all problems. {\blue All matrices are preliminary
diagonally scaled before solving the corresponding linear system}.
{\blue
We consider as the exact solution a vector with all ones and computed the right hand side accordingly.
Unless differently stated, we stop the PCG iteration as soon as the relative residual norm is below $\texttt{tol} = 10^{-8}$.}

\subsection{Sequential tests}
As common when dealing with polynomial preconditioners, the main issue is to cheaply assess the extremal
eigenvalues. In the numerical results reported below we approximated $\beta_0$ with few iterations
of the power method and $\alpha_0$ with the non preconditioned DACG method \cite{bgp97nlaa} up
to $10^{-2}$ tolerance on the relative residual.  The sequential tests have been performed using Matlab on
on an Intel Core 2 Quad at 3.50GHz, each core being equipped with 16Gb RAM.

The results reported in Table \ref{NewCheb} refer to matrices \texttt{Opt\_Transp}, 
\texttt{Lap1600} and \CC.

\begin{table}[!ht]
	\caption{Results of the NC polynomial preconditioner for matrices \texttt{Opt\_Transp},
	\texttt{Lap1600}  and \CC for various degrees of the polynomial preconditioner. For the \CC matrix
	the tolerance was set to $10^{-12}$.}
	\label{NewCheb}
\begin{center}
\begin{tabular}{r|rrrr|c|rr|c|rrr}
&\multicolumn{5}{c|}{Matrix  \texttt{Opt\_Transp}} & 
	\multicolumn{3}{c|}{Matrix  \texttt{Lap1600}}  &
	\multicolumn{3}{c}{Matrix  \CC}  \\
	\hline
	$m$          	& iter & ddot  & $A\times \fv$ & $\|\fr_k\|/\|\fb\|$  & CPU(s)
	& iter & $\|\fr_k\|/\|\fb\|$  
	& CPU(s) & iter & CPU(s) &
	$\|\fr_k\|/\|\fb\|$   \\
\hline
0   & 3433   & 10299 & 3433 & 9.85e-09  &  26.39 & 4517   & 9.75e-09  & 203.52 & 9037 & 3040.0 & 9.94e-13 \\
1   & 1773   & 5319 & 3536 &  9.70e-09  &  23.25 & 2313   & 9.99e-09 &  176.58 & 4604 & 2990.9 & 9.98e-13\\
3   &  879   & 2637 & 3516 &  9.92e-09  &  20.95 & 1174   & 9.87e-09 &  160.50 & 2413 & 3044.7 & 9.96e-13\\
7   &  439   & 1317 & 3512 &  8.85e-09  &  19.86 &  589   & 9.40e-09 &  151.79 & 1204 & 2999.6 & 9.52e-13\\
15  &  222   & 666  & 3552 &  8.39e-09  &  19.59 &  295   & 9.92e-09 &  147.83 &  604 & 2988.5 & 9.66e-13\\
31  &  117   & 351  & 3744 &  7.78e-09  &  20.33 &  149   & 9.50e-09 &  146.80 &  304 & 2996.9 & 9.97e-13 \\
63  &   69   & 207  & 4416 &  5.49e-09  &  23.85 &   77   & 7.09e-09 &  151.17 &  156 & 3069.9 & 8.92e-13
\end{tabular}
\end{center}
\end{table}

\noindent
Some comments are in order. The good news are that, 
	apart from an obvious decrease of the number of scalar products:
\begin{enumerate}
	\item Assessment of extremal eigenvalues is relatively cheap. It took only 0.69 seconds for the \texttt{Opt\_Transp} matrix,  1.33 seconds for the Laplacian and 
	$17.4$ seconds for the \CC matrix.
\item The norm of the {\em true} residual at convergence decreases with $m$, confirming the 
improved conditioning of the preconditioned matrix.
\item The CPU time decreases by 15\% -- 25\% by increasing the polynomial degree from $m=0$ to $m=15$.
	This does not hold for the matrix \CC for which the cost of the matrix-vector products is predominant over
		the scalar products due to the high number of average nonzeros per row.
\end{enumerate}

\smallskip
{\blue

\noindent
{\bf Remark}. We do not claim that our polynomial preconditioner
can compare favorably with other well-known sequential accelerators such as the Incomplete Cholesky preconditioner.
We report, however, the performance of this preconditioner (as implemented by the Matlab function
\texttt{ICHOL}($\delta$), $\delta$ being the drop tolerance) in combination with the CG solver for the three analyzed matrices. We also report the density
of the Cholesky factor as $\rho = $ nonzero($L$)/nonzero($A$) (which is a measure of the increased storage demand of this preconditioner).

\begin{center}
\begin{tabular}{l|l|rrr}
	matrix & $\delta$ & $\rho$ & Iter & CPU \\
	\hline 
	\texttt{Lap1600}& no fill & 0.5& 1344  &   151.02 \\
	\texttt{Opt\_Transp}  & $10^{-4}$  &1.87 & 201 & 7.32\\
	\CC & $10^{-4}$ & \multicolumn{3}{c}{negative pivot encountered} \\
	\CC & $10^{-5}$ & \multicolumn{3}{c}{out of memory} \\[.3em]
\end{tabular}
\end{center}

\noindent
Number of iterations and CPU times are smaller than with the polynomial preconditioner, which, by contrast, 
does not require additional memory, is completely matrix free and easily parallelizable. Moreover
we could not compute the IC factorization of the larger matrix \CC due to memory limitations.
}
\subsection{Numerical Results on a Parallel Platform}
	The polynomial preconditioner is based on matrix-vector products and no scalar products. 
This feature can be successively exploited on parallel architectures since, as known,
when a high number of processors is employed, 
the dot product, being the only task that involves a collective communication, 
reveals a bottleneck for the parallel efficiency.

An efficient implementation of a parallel matrix vector product is obviously  mandatory to achieve high parallel efficiency.
In this paper we use an improved  MPI-Fortran routine as successfully experimented in \cite{mbcv06}.
We used a block row distribution of the coefficient matrix with
complete consecutive rows assigned to different processors.

All tests have been performed on the new HPC Cluster Marconi
at the CINECA Centre, on both the A1 version (1512 nodes,  2 $\times$ 18-cores Intel Xeon E5-2697 v4 (Broadwell) at 2.30 GHz)
and the more recent A2 update (with 3600 nodes and 1$\times$ 68-cores
Intel Xeon 7250 CPU (Knights Landing) at 1.4GHz).
The Broadwell nodes have 128 Gb memory each, while  in the A2 system the RAM is subdivided
into 16GB of MDRAM and 96GB of DDR4.
The Marconi Network type
is: new Intel Omnipath, 100 Gb/s. (MARCONI is the largest Omnipath cluster of the world).

Throughout the whole section we will denote with $T_\p$ the CPU elapsed times
expressed in seconds (unless otherwise stated)
when running the code on $\p$ processors.
We include a relative measure of the parallel
efficiency achieved by the code. To this aim we will  denote as
$S_\p^{(n_0)}$, the pseudo speedup computed with respect to the smallest number of processors ($n_0$)
used to solve the given problem:
\[ S_\p^{(n_0)} = \frac{T_{n_0} \ n_0}{T_\p}. \]
We will denote $E_\p^{(n_0)}$ the corresponding relative parallel efficiency, obtained according to
\[E_\p^{(n_0)}= \dfrac{S_\p^{(n_0)}}{\p}  = \frac{T_{n_0}\  n_0}{T_\p {\p}} . \]

	\begin{table}[h!]
		\caption{Scalability analysis for the \emi matrix.}
	\label{emilia}
	\begin{center}
	\begin{tabular}{r|rrr|rrr|rrr|c}
		& \multicolumn{3}{c}{\texttt{nlev} $ = 5$}& \multicolumn{3}{c}{\texttt{nlev} $ = 2$}& \multicolumn{3}{c}{\texttt{nlev} $ = 0$} \\
		\hline
		& & & & & & & & &\\[-.3em]
		$\p$ & iter & $T_p$ & 		$E_p^{(16)}$ 
		& iter & $T_p$ & 		$E_p^{(16)}$ 
		& iter & $T_p$ & 		$E_p^{(16)}$ & $\dfrac{T_p(\texttt{lev} =0)}{T_p(\texttt{lev} =5)}$ \\[.5em]
		\hline
		16 & 379 & 114.44 && 3008 & 115.64 && 11386 & 117.15 && 1.02 \\
		64 & 379 &  33.33 &86\%& 3008 &  34.43 &84\%& 11382 &  37.74 &78\% & 1.13\\
		256 & 379 &  10.39 &69\%& 3008 &  12.35 &59\%& 11380 &  16.75 &44\% & 1.61\\
		512 & 379 &   6.15 &58\%& 3008 &   9.15 &39\%& 11380 &  14.70 &25\% & 2.35\\
	\end{tabular}
	\end{center}
\end{table}

In Table \ref{emilia} we report the scalability results for matrix \emi using levels
$0, 2$ and $5$ which correspond to using a polynomial preconditioner of degree $0$, $3$ and $31$, respectively.
It is shown that the parallel efficiency is greatly improved when a high degree of the preconditioner is used.
The relative efficiency from 16 to 1024 processors is increased from $25\%$ ($\text{lev}  = 0$) 
to $58\%$ ($\text{nlev}  = 5$) by a factor 2.35.

\noindent
The scalability results for matrix {\tt Cube5317k}, reported in Table \ref{cube} show a 1.6 CPU time
reduction from \texttt{nlev} $ = 0$ to  \texttt{nlev} $ = 5$.
\begin{table}[h!]
	\caption{Scalability analysis for the {\tt Cube5317k} matrix.}
	\label{cube}
	\begin{center}
	\begin{tabular}{r|rrr|rrr|c}
		& \multicolumn{3}{c}{\texttt{nlev} $ = 5$}& \multicolumn{3}{c}{\texttt{nlev} $ = 0$} \\
		\hline
		& & & & & & &\\[-.3em]
		$\p$ & iter & $T_p$ & 		$E_p^{(64)}$ 
		& iter & $T_p$ & 		$E_p^{(64)}$ & 
		$\dfrac{T_p(\texttt{nlev} =0)}{T_p(\texttt{nlev} =5)}$ \\[.5em]
		\hline
		64 & 298 & 154.79 &-- &9038 & 164.7 &-- & 1.06\\
		128 & 298 &  85.33 &91\%& 9038 &  91.30 &90\% & 1.07\\
		256 & 298 &  46.60 &83\%& 9038 &  53.63 &77\% & 1.15\\
		512 & 298 &  28.04 &69\%& 9038 &  35.12 &59\% & 1.25\\
		1024 & 298 & 21.23  &46\% &9038 &  33.94 &30\% & 1.60\\
	\end{tabular}
	\end{center}
\end{table}

\noindent
The different parallel performance is related to the nonzero patterns of the two matrices. In matrix 
{\tt Cube5317k} the nonzeros are more spread far from the diagonal (as a result of a local mesh refinement). 
This implies that a given processor
must receive/send data with a large number of other processors when performing the matrix-vector
product. This behavior is clearly shown in Figure \ref{npact}.
For the {\tt Cube5317k} matrix the predominant parallel cost is represented by the matrix-vector product
which is the bottleneck of the parallel computation for a high number of processors.
{\blue 
Clearly, this unvaforable sparsity pattern can be improved by preprocessing the linear system
with a suitable graph partitioning and fill-reducing matrix ordering. However we consider
this test case, as it is, a worst case scenario for our preconditioner, which, however,  
is shown to obtain satisfactory speed-ups}.

\begin{figure}[h!]
	\vspace{-5.3cm}
	\begin{center}
	\includegraphics[width=.9\textwidth]{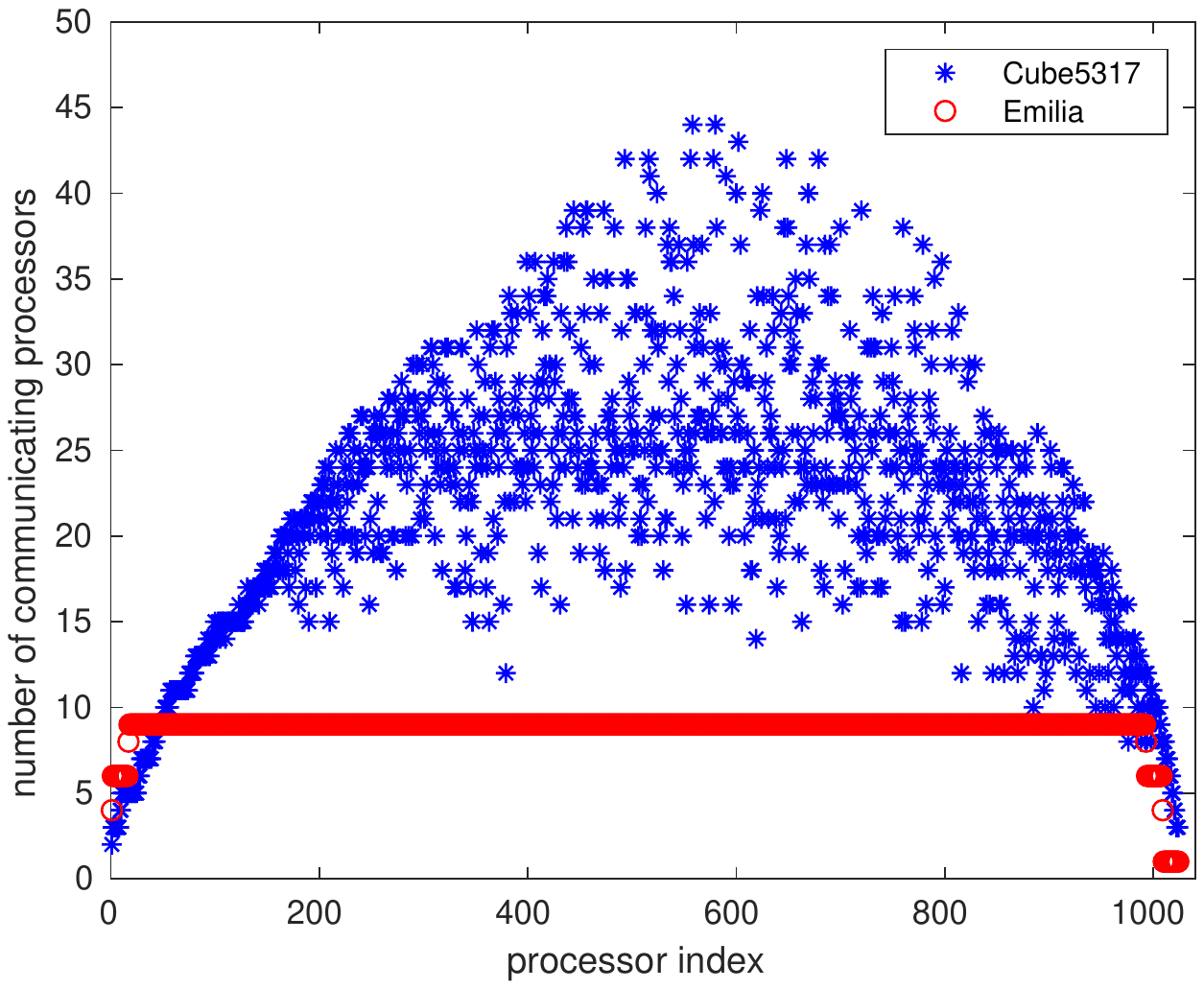}
	\vspace{-5.3cm}
	\end{center}
	\caption{Number of communicating processors with a given processor in performing
	the matrix-vector product.
	Matrices {\tt Cube5317k} and \emi with $\p$ = 1024.}
	\label{npact}
\end{figure}

\subsection{Results on huge matrices}
We now report the results in solving huge linear systems arising from Finite Difference (FD) 3D discretization of the 
Poisson equation in the unitary cube. These last runs have been conducted on the new Marconi 100 supercomputer available
at Cineca. 
MARCONI 100 is the new accelerated cluster based on 
980 IBM Nodes, each equipped with 
2x16 cores IBM POWER9 AC922 at 3.1 GHz  processors.

We consider three very large matrices: \texttt{lap3d(nx)}, where $\texttt{nx} = 512, 1024, 2048$ is the number of subdivisions
in each spatial dimension. The size, nonzeros and condition number of these matrices are reported in Table \ref{sizeFD}.
		\begin{table}[h!]
			\caption{Size $n$, number of nonzeros \texttt{nnz} and condition number $\kappa $ for the 3 FD-3D matrices.}
\label{sizeFD}
\begin{center}
\begin{tabular}{r|clc}
	\texttt{nx} & $n$ &  \texttt{nnz} & $\kappa(A)$ \\
\hline  
	512  & $1.3 \times 10^8$ &  $9.4 \times 10^9$   & $1.06 \times 10^5$ \\
	1024 & $1.1 \times 10^9$ &  $7.5 \times 10^9$   & $4.24 \times 10^5$ \\
	2148 & $8.6 \times 10^9$ &  $6.0 \times 10^{10}$& $1.70 \times 10^6$ \\
\end{tabular}
\end{center}
		\end{table}

\begin{table}[h!]
\begin{center}
	\caption{CPU times and iterations for the  \texttt{FD-3D(\texttt{nx})} problems for various
	degree of the polynomial and varying number of processors.}
	\label{lap3d}
\begin{tabular}{l|r|rr|rr|rr|l}
	\texttt{nx} & 	$\p$ &  \multicolumn{2}{|c}{\texttt{nlev} $ = 5$}& \multicolumn{2}{|c}{\texttt{nlev} $ = 2$} & \multicolumn{2}{|c}{\texttt{nlev} $ = 0$}&  
		\multirow{2}*{$\dfrac{T_p(\texttt{nlev} =0)}{T_p(\texttt{nlev} =5)}$} \\
		 & & iter & $T_p$  & iter & $T_p$  & iter & $T_p$  & \\
	\hline
	512 &	64 & 45 &   67.0  &   325 &  67.4   &  1300 &  95.3  & 1.4 \\
	&128 & 45 &   36.2 &    325 &  38.1   &  1300 & 50.2  & 1.4 \\
	&256 & 45 &   21.8 &    325&  21.8 &  1300 &  27.7  & 1.3 \\
	&512  & 45 &  13.8  &    325 & 13.3   &  1300 &  16.8  &  1.3 \\
	\hline
	1024 &64 & 88 &  858.4 &    637 &  945.2 &  2553 & 1481.7 & 1.7 \\
	&256 & 88 &  254.3 &    637 & 284.3 &  2553 &  400.6 & 1.6 \\
	&1024 & 88 &  97.2 &    637 & 101.6  &  2553 &  131.5 & 1.4 \\
	\hline
	2048 &512  & 165 & 1925.7 &-- &-- &    5033 & 3169.8 & 1.6\\
	&2048 & 165 &  710.5 &-- &--  &   5033 & 1001.5 & 1.4 \\
\end{tabular}
\end{center}
\end{table}
	The results, reported in Table \ref{lap3d},  show that we are able to solve very huge size problems with a good (relative) strong scalability.
	Moreover
the polynomial preconditioner (either with \texttt{nlev} $=2$ or \texttt{nlev} = $5$) takes from 1.3 to 1.7 
	less CPU time  than the diagonal preconditioner.

On the huge problem \texttt{lap3d(2048)} the relative efficiency from 512 to 2048 processors is around
	$70\%$.  This problem, with eight billion unknowns and 56 billion nonzeros has been
	solved with 165 iterations, three times as many scalar products, and 710.5 seconds
	with 2048 processors.

\textbf {Weak scalability analysis}.  We finally perform a sort of weak scalability analysis, weighted by taking into
	account that the condition number, and hence the number of PCG iterations, grows with $\texttt{nx}$.
	In detail, doubling the \texttt{nx} parameter
		the size of the corresponding matrix increases by a factor 8; moreover  its condition number increases
		by a factor 4 and therefore the PCG iteration number is expected to roughly double. 
	Summarizing,	 from a matrix to the subsequent
		one in the sequence, we may expect an increase of a factor 16 in the CPU time (with constant number of processors).
		Defining as
		$T_{nx,\p} $ the CPU time needed to solve a FD-3D matrix with \texttt{nx} and $\p$ processors 
a perfect weak scalability would predict a dependence of the CPU time on \texttt{nx} and $\p$ as 
\[ T_{\texttt{nx},\p} = O \left(\frac{\texttt{nx}^4}{p}\right) \] from which, assuming now $ \p = \texttt{nx}$:
	\[ T_{2\p, 2\p} = 8 T_{\p,\p} = 64 T_{\frac{\p}{2}, \frac{\p}{2}}.\]
From Table \ref{lap3d} we have indeed that, for \texttt{nlev} = $0$,
$\dfrac{T_{2048,2048}}{T_{512,512}} = 59.6$  whereas for \texttt{nlev} = $5$ 
	$\dfrac{T_{2048,2048}}{T_{512,512}} = 51.5$, which are both smaller (and hence better) than
	the theoretically optimal value of 64. 
	{\blue 
	\subsection{Comparisons with other parallel preconditioners}
	The proposed preconditioner has many pleasant features such as: No additional memory requirements, No need
	to explicitly store the matrix, It takes the number of scalar products to a very low value. 
	To show that it is also convenient in terms of overall efficiency we carried out a comparison with a state-of-the-art
	parallel preconditioned solver for SPD linear system. It is the solver \texttt{chronos}, available
	at the webpage \texttt{https://www.m3eweb.it/chronos/}, which makes use of an  enhanced AMG solver,
	partially based on a FSAI smoother with dynamical nonzero pattern  selection \cite{FMMSJ19, MFJ19}

	In Table \ref{AMG} we reported the results in solving  the FD matrix with $nx = 512$ for the PCG
	method accelerated with either the AMG or the FSAI preconditioners, after some trials to select the optimal parameters. Since the setup time to evaluate the preconditioner is rather high for this approach we reported this
	in the table as $T_{\text{setup}}$ while the CPU time for the PCG solution
	is $T_{\text{solver}}$. $T_p = T_{\text{setup}} + T_{\text{solver}}$ is, as before, the overall CPU time.

	\begin{table}[h!]
		\caption{Results for the solution of the FD-3D problem with $nx = 512$ using the \texttt{chronos} package.}
		\label{AMG}
		\begin{center}
	\begin{tabular}{r|rrrr|rrrr}
		\hline
		& \multicolumn{4}{c|}{AMG preconditioner } &
		  \multicolumn{4}{c}{FSAI preconditioner} \\
		& \multicolumn{4}{c|}{with FSAI as smoother} &&&&\\
		  \hline
		$\p$ & iter  &  $T_{\text{setup}}$  & $T_{\text{solver}}$ & $T_p$      
		     &       &  $T_{\text{setup}}$  & $T_{\text{solver}}$ & $T_p$      \\
		     \hline
		64&23&21.5&10.8&32.3 &786&5.9&86.4&92.4 \\
128&24&15.3&7.1&22.4 &        774&2.7&43.7&46.5 \\
256&26&12.4&4.4&16.8&       782&1.7&24.4&26.2 \\
512&28&14.3&4.6&18.8&       758&0.7&13.8&14.5 \\

	\end{tabular}
		\end{center}
	\end{table}
	Inspection of Tables \ref{AMG} and  \ref{lap3d} reveals that our polynomial preconditioner compares very well with
	this state-of-the-art solver both in terms of scalability and CPU times.
	Regarding the PCG solution times only, the AMG approach outperforms the NC preconditioner, however the gap progressively
	reduces as the number of processors increases.

	}

\section{Conclusions}
\label{conc}
We have proposed a (potentially high-degree) polynomial preconditioner for the Conjugate Gradient method with the aim of greatly reducing the number of scalar products which may represent a bottleneck especially in parallel computations.  
By avoiding clustering of extremal eigenvalues, the preconditioner
	obtains its best performances when the degree $m$ is relatively high (good results
	have been obtained with  $m = 31$ or $m = 63$).
	Numerical results onto very large matrices reveal that these polynomial preconditioners
may be successfully employed to accelerate the Conjugate Gradient method by drastically
	reducing the number of scalar products (and hence the collective communications in parallel
	environments). 
 In sequential computations the polynomial preconditioner with degree $31$ reduces the CPU time 
 of about $30\%$ with respect to the diagonal preconditioner.
 Parallel runs with up to 2048 processors on the Marconi supercomputer show that the important reduction in the number of scalar products (which 
 reduces roughly to 97\% smaller with respect to the diagonal preconditioner, with $m = 31$)
 yielding a improvement over the diagonal preconditioner from 30\% to 60\% of the total CPU time.

	Further study is undergoing to give theoretical setting how to compute
	the optimal scaling parameter. Moreover, a low-rank acceleration of the polynomial preconditioner
	will be investigated, following e.g. \cite{LB_Algorithms_2020} by exploiting the
	well separation of the smallest eigenvalues provided by our polynomial preconditioner.
	{\blue 
	We finally observe that the described approach can be applied whenever
	a first level parallel preconditioner is at hand in factored form, say $P_0 = W W^T$, to obtain a second level
	preconditioner applying the Newton-Chebyshev polynomials to the matrix $ W^T A W$.}

\subsection*{Acknowledgements}
This work was partially supported by the Project granted by the CARIPARO
foundation {\em Matrix-Free Preconditioners for Large-Scale Convex
Constrained Optimization Problems (PRECOOP)} and
by the INdAM Research group GNCS, 2020 Project: Optimization and advanced linear algebra for problems arising from PDEs.

\end{document}